\documentclass{amsart}
\usepackage{verbatim,amssymb,amsmath,amscd,latexsym,amsbsy,mathrsfs}
\usepackage{enumerate}
 \input{xy}
\xyoption{all}

\newtheorem{thm}{Theorem}[section] \newtheorem{pro}[thm]{Proposition}
\newtheorem{lemma}[thm]{Lemma}
\newtheorem{cor}[thm]{Corollary}
\numberwithin{equation}{section}

\newtheorem{question}[thm]{Question}
\theoremstyle{remark}

\theoremstyle{definition} 
\newtheorem{rmk}[thm]{Remark} 

\newtheorem{example}[thm]{Example}
\newtheorem{df}[thm]{Definition}

\DeclareMathAlphabet{\mathpzc}{OT1}{pzc}{m}{it}
 \DeclareMathOperator*{\Hom}{Hom}
 \DeclareMathOperator*{\im}{Im}
 
\DeclareMathOperator*{\Gal}{Gal}

\DeclareMathOperator*{\Char}{char}
\DeclareMathOperator*{\lcm}{lcm}

\DeclareMathOperator*{\supp}{Supp}
\DeclareMathOperator*{\degram}{dr}
\DeclareMathOperator*{\parslope}{par-\mu}

\DeclareMathOperator*{\image}{im}

 \newcommand{\QQ}{\mathbb{Q}}
 \newcommand{\Aff}{\mathbb{A}}
\newcommand{\PP}{\mathbb{P}} \newcommand{\FF}{\mathbb{F}}

\newcommand{\cO}{\mathcal{O}} 
\newcommand{\scrD}{\mathcal{D}}\newcommand{\scrZ}{\mathcal{Z}}
\newcommand{\scrX}{\mathcal{X}}\newcommand{\scrY}{\mathcal{Y}}
\newcommand{\scrW}{\mathcal{W}}\newcommand{\scrC}{\mathcal{C}}

\newcommand{\cD}{\mathpzc{D}}
 
\newcommand{\cK}{\mathcal{K}}

\newcommand{\onto}{\twoheadrightarrow}

\begin{document}
\title{Formal orbifolds and orbifold bundles in positive characteristic} 
 \author{
  Manish Kumar
 }
 \address{
Statistics and Mathematics Unit\\
Indian Statistical Institute, \\
Bangalore, India-560059
  }
  \email{manish@isibang.ac.in}
  
  \author{A. J. Parameswaran}
  \address{
  Tata Institute of fundamental research \\
  Mumbai
  }
\begin{abstract}
  We define formal orbifolds over an algebraically closed field of arbitrary characteristic as curves together with some branch data. Their \'etale coverings and their fundamental groups are also defined. These fundamental group approximates the fundamental group of an appropriate affine curve. We also define vector bundles on these objects and the category of orbifold bundles on any smooth projective curve. Analogues of various statement about vector bundles which are true in characteristic zero are proved. Some of these are positive characteristic avatar of notions which appear in the second author's work (\cite{parabolic}) in characteristic zero.
 \end{abstract}
\maketitle
\section{Introduction}

In this article we consider curves $X$ over an algebraically closed field with some branch data $P$ on $X$ which we call formal orbifolds (see Definition \ref{main-definition}). These object $(X,P)$ when considered over fields of characteristic zero has a simpler description and were studied in \cite{parabolic}. Though many of the results in this paper are characteristic free, they were mostly known in characteristic zero. 

We define morphisms of formal orbifolds and study homotopic behaviour of these objects. It is shown that the \'etale fundamental group $\pi_1(X,P)$ are profinite groups of finite rank (Corollary \ref{pi_1-finite-rank}) and the \'etale fundamental group of any affine curve (which is of infinite rank and a mysterious object) can be approximated by the fundamental group of its smooth completion with various branch data supported on the boundary (Theorem \ref{pi_1-approx}). Hence these objects help in understanding the fundamental group of affine curves. Formal orbifolds are homotopically closer to smooth projective curves than affine curves. Moreover $(X,O)$ where $O$ is the trivial branch data is same as the projective curve $X$. So one could think of formal orbifolds as a generalization of smooth projective curves and see which results for smooth projective curves could be generalized to formal orbifolds. For instance, we prove a version of Riemann-Hurwitz formula (Theorem \ref{parabolic-RH}). 

If the branch data of a formal orbifold ``comes'' from a covering of smooth projective curves then it is called geometric and the resulting formal orbifold is called a geometric formal orbifold. We investigate which formal orbifolds are geometric formal orbifolds. Though we provide some partial results, a complete characterization is out of reach at the moment.

The second part of the paper (Section \ref{sec:orbifolds}) deals with vector bundles on a geometric formal orbifold $(X,P)$. These are defined to be $\Gamma$-bundles on $Y$ where $(Y,O)\to (X,P)$ is an \'etale $\Gamma$-Galois cover for some finite group $\Gamma$. We show that this is a well-defined category (independent of the choice of the cover $Y$) (Proposition \ref{parabolic-bundles-well-defined}) consisting of duals and tensor products. The subcategory of degree zero strongly semistable bundles (respectively essentially finite bundles) form a Tannakian category (Theorem \ref{tannakian}). Hence the notion of Nori fundamental group extends to geometric formal orbifolds as well (Remark \ref{Nori-pi_1}).

Orbifold bundles on a smooth projective curve $X$ is defined to be an object in the category of vector bundle over geometric formal orbifold $(X,P)$ where $P$ is a geometric branch data on $X$. A version of the projection formula for orbifold bundles is also proved (Theorem \ref{projection-formula}).
 
\section{Formal orbifolds and coverings}
Let $X$ be a smooth projective curve over an algebraically closed field $k$. Let $x\in X$ be any closed point.  Let $\cK_{X,x}$ denote the fraction field of the 
completion $\hat \cO_{X,x}$ of regular functions at $x$.

\begin{df}\label{main-definition}
\begin{enumerate}
 \item 
  A quasi-branch data on $X$ is a function $P$ which sends a closed point $x$ of $X$ to a finite Galois extension $P(x)$ of $\cK_{X,x}$ in some fixed algebraic closure of $\cK_{X,x}$. Let $P$ and $P'$ be two 
	\emph{quasi-branch data} on $X$, we say $P\le P'$ if $P(x)\subset P'(x)$ for all closed points $x\in X$.  
\item 
The \emph{support of $P$}, $\supp(P)$ is defined to be the set of all $x\in X$ such that $P(x)$ is a non trivial extension of $\cK_{X,x}$. A quasi-branch data $P$ is said to be a \emph{branch data} if $\supp(P)$ is a finite set. 
 \item 
 A smooth projective curve with a branch data is called a \emph{formal orbifold} . Since fiber products are not connected in general, we allow formal orbifolds to be disconnected.
\end{enumerate}
\end{df}

The branch data on $X$ with empty support is denoted by $O$ and is called the \emph{trivial branch data}.

\begin{rmk}
A formal orbifold $\scrX=(X,P)$ is equivalent to the following data: $\scrX=(X, \{x_i,P(x_i): 1\le i\le r\})$ consist of a smooth projective curve $X$ together with 
finitely many closed points $x_1,\ldots,x_r$ and finite Galois extensions $P(x_i)/\cK_{X,x_i}$ associated to $x_i$ for $1\le i\le r$.
\end{rmk}

\begin{lemma}
Let $X$ be a smooth projective curve and $P$, $P'$ be two quasi-branch data on $X$. For any closed point $x\in X$, let $(P\cap P')(x)= P(x)\cap P'(x)$ and 
$(PP')(x)=P(x)P'(x)$. Then $P\cap P'$ and $PP'$ are quasi-branch data. If $P$ or $P'$ is a branch data then so is $P\cap P'$. If $P$ and $P'$ are branch data then so is $PP'$. Also $P\le PP'$ and $P\cap P'\le P$   
and $P'$.
\end{lemma}

\begin{proof}
It follows immediately from the fact that intersection or compositum of Galois extensions are Galois and the observation that $\supp(P\cap P')\subset \supp(P)\cap \supp(P')$ and $\supp(PP')=\supp(P)\cup\supp(P')$.  
\end{proof}

\begin{rmk}(Branch data associated to a cover)
Let $f:Y\to X$ be a finite morphism of smooth projective curves. For $x\in X$ a closed point, define $B_f(x)$ to be the compositum of Galois closures of $\cK_{Y,y}/\cK_{X,x}$ 
over all $y\in Y$ such that $f(y)=x$. The compositum is taken in a fixed algebraic closure of $\cK_{X,x}$. Since branch locus of a covering is a finite set $B_f$ is a 
branch data  with $\supp(B_f)$ same as the branch locus of $f$. 
\end{rmk}

\begin{df}
Let $f:Y\to X$ be a covering of smooth projective curves. Let $P$ be a branch data on $X$. For $y\in Y$, we define the \emph{pull-back of a branch data} as $f^*P(y)=P(f(y))\cK_{Y,y}$ as an 
extension of $\cK_{Y,y}$. Since $P(f(y))/\cK_{X,f(y)}$ is a Galois extension, so is $f^*P(y)/\cK_{Y,y}$. Moreover, $\supp(f^*P)\subset f^{-1}(\supp(P))$.
\end{df}

\begin{df}
\begin{enumerate}
\item
Let $Y$ and $X$ be smooth projective curves with branch data $Q$ and $P$ respectively. A morphism of formal orbifolds $f:(Y,Q)\to (X,P)$ is a morphism 
$f:Y\to X$ such that for all $y\in Y$ the extension $Q(y)/\cK_{X,f(y)}$ contains the extension $P(f(y))/\cK_{X,f(y)}$. \\
\item
We say that the morphism of formal orbifolds $f:(Y,Q) \to (X,P)$ is \'etale at $y$ if $Q(y)=P(f(y))$. 
And we say it is \'etale if it is \'etale for all closed point $y\in Y$.  \\
\item
A finite cover $f:Y\to X$ is said to be an essentially \'etale cover of $\scrX=(X,P)$ if for any closed point $y\in Y$ the field $\cK_{Y,y}$ viewed as an extension of $\cK_{X,f(y)}$ 
is contained in $P(f(y))$.
\end{enumerate}
\end{df}

\begin{rmk}
 A finite cover $Y\to X$ is an essentially \'etale cover of $\scrX$ then its Galois closure $\tilde Y\to X$ is also an essentially \'etale cover of $\scrX$.
\end{rmk}

\begin{rmk}\label{id-morphism}
Let $P$, $P'$ be two branch data on $X$. The identity $id_X$ defines a morphism of formal orbifolds $(X,P)\to (X,P')$ iff $P\ge P'$. And the morphism is 
\'etale iff $P=P'$.
\end{rmk}

\begin{pro}
If the $\Char(k)$ is 0 then the above definitions of formal orbifold and essentially \'etale cover agree with the definitions of parabolic curve and parabolic cover respectively in \cite{parabolic}.
\end{pro}

\begin{proof}
If $\Char(k)=0$ then by Newton's theorem every finite extension $L/k((t))$ is a cyclic extension of the order same as the degree of the extension and conversely every 
positive integer $n$ defines a unique cyclic extension of order $n$. Hence fixing field extensions at various points of $X$ is equivalent to fixing non-negative integers. 
The definition of essentially \'etale covers also agree with parabolic covers because $L'/k((t))$ is an intermediate extension iff $[L':k((t))]$ divides $[L:k((t))]$.
\end{proof}

\begin{pro}\label{product-parabolic}
Let $Y\to X$ and $Z\to X$ be essentially \'etale covers of $\scrX=(X,P)$. Then  the normalization of the fiber product $\widetilde{Y\times_X Z}\to X$ is an essentially \'etale cover of $\scrX$.
\end{pro}

\begin{proof}
We may assume $Y\to X$ and $Z\to X$ are Galois. Let $W$ be any component of $\widetilde{Y\times_X Z}$. Let $\tau=(y,z)\in W$ be a point lying above $x$. Then $y\in Y$ and 
$z\in Z$ lie above $x$. It follows from \cite[Theorem 3.4]{killwildram} that $\cK_{W,\tau}\subset P(x)$.
\end{proof}

\begin{cor}\label{inverse-system}
The set of essentially \'etale covers of a formal orbifold $\scrX$ form an inverse system.
\end{cor}
 
\begin{proof}
 Proposition \ref{product-parabolic} implies that the underlying partially ordered set is a directed set.
\end{proof}

\begin{lemma}\label{parabolic-etale}
Let $(X,P)$ be a formal orbifold and $f:Y\to X$ be a covering. Then for any branch data $Q$ on $Y$, $f$ defines a morphism of formal orbifolds $f:(Y,Q)\to (X,P)$ iff $Q\ge f^*P$. Moreover, $f:(Y,f^*P) \to (X,P)$ is an \'etale morphism of formal orbifolds iff $f:Y\to X$ is a essentially \'etale cover of $(X,P)$.
\end{lemma}

\begin{proof}
 Note that $(f^*P)(y)=P(f(y))\cK_{Y,y}$. So for a branch data $Q$ on $Y$ and $y\in Y$, $Q(y) \supset P(f(y))$ iff $Q(y) \supset P(f(y))\cK_{Y,y}=(f^*P)(y)$. Hence $f:(Y,Q)\to (X,P)$ is a morphism iff $Q\ge f^*P$. 
 
 Moreover $f:(Y,f^*P)\to (X,P)$ is an \'etale morphism iff $(f^*P)(y)=P(f(y))$ for all $y\in Y$, i.e. $\cK_{Y,y}\subset P(f(y))$ for all $y\in Y$. But this is the definition of $f:Y\to X$ to be essentially \'etale cover of $(X,P)$.
\end{proof}

\begin{cor} \label{galois-covers-R_f-parabolic-etale}
 Let $f:Y\to X$ be a finite Galois cover of smooth projective curves. Then $f:(Y,O)\to (X,B_f)$ is an \'etale morphism of formal orbifolds.
\end{cor}

\begin{proof}
 Let $x\in X$ be a closed point and $y\in f^{-1}(x)$. Since $f:Y\to X$ is a Galois cover, the extension $\cK_{Y,y}/\cK_{X,x}$ is a Galois extension which is independent of the choice of $y\in f^{-1}(x)$. Hence $B_f(x)=\cK_{Y,y}$. By definition $f:Y\to X$ is an essentially \'etale cover of $(X,B_f)$ and $f^*B_f$ is the trivial branch data. Hence by above lemma $f:(Y,O)\to (X,B_f)$ is \'etale.
\end{proof}

The next proposition shows that the fiber products of formal orbifolds exist. Note that here by fibre product we mean the normalized fibre product. In fact it is easy to check that normalized fibre product
satisfies the universal property of the fibre product among orbifolds. 

\begin{pro}\label{fiber-product}
Let $f:(Y,Q)\to (X,P)$ and $g:(Z,R)\to (X,P)$ be morphism of formal orbifolds. Let $W$ be the normalization of $Y\times_X Z$  
and $F(w)=Q(y)R(z)/\cK_{W,w}$ where $w=(y,z)\in W$. Then $(W,F)$ is the fiber product in the category of formal orbifolds.
\end{pro}

\begin{proof}
Note that $\cK_{Y,y}\cK_{Z,z}=\cK_{W,w}$ as extension of $\cK_{X,x}$ by \cite[Theorem 3.4]{killwildram}. The compositum $Q(y)R(z)$ can be taken in any fixed algebraic closure of $\cK_{W,w}$. As $Q(y)/\cK_{Y,y}$ and $R(z)/\cK_{Z,z}$ are Galois $F(w)/\cK_{W,w}$ is also Galois and the hence the extension is independent of the choice of the algebraic closure.  

One easily checks that projection maps are morphisms of formal orbifolds. Finally, $(W,F)$ satisfies the universal property for fiber product follows from the universal property of the usual fiber product and the universal property the normalization.
\end{proof}

\begin{example}
Let $X$ be a smooth projective curve and $P,P'$ be branch data on $X$. Let $B$ be any branch data on $X$ such that  $B\le P\cap P'$. Then $(X,P)\times_{(X,B)}(X,P')=(X,PP')$. 
\end{example}

\begin{pro}\label{etale-pullback}
Pullback of \'etale morphism of formal orbifolds is \'etale. 
\end{pro}

\begin{proof}
Let $\scrY\to \scrX$ be an \'etale morphism and $\scrZ\to \scrX$ be a morphism of formal orbifolds. Let $\scrW$ be the fiber product and $\scrW\to \scrZ$ be 
the projection map. The result now follows from Lemma \ref{parabolic-etale} and the definition of the branch data on $\scrW$ as described in Proposition \ref{fiber-product}.
\end{proof}

\subsection{Hurwitz formula}

We will define the genus of the formal orbifold $(X,P)$ using the noion of different (\cite{Serre-local.fields}) and obtain a Riemann-Hurwitz formula.

Let $L/k((t))$ be a finite extension. Let $\scrD(L/k((t)))$ be the different of the extension $L/k((t))$ and let 
$\degram(L/k((t)))=v_L(\scrD(L/k((t))))$ where $v_L$ is the valuation of $L$ (with the value of the uniformizer of $L$ being 1). For a Galois extension 
$L/k((t))$ with Galois group $G$, let $G=G_0\ge G_1\ge G_2 \ldots$ be the lower ramification filtration. Recall that by Hilbert's different formula (\cite[Chapter 4, \S, Proposition 4]{Serre-local.fields})  
$$\degram(L/k((t)))=\sum_{i=0}^{\infty}(|G_i|-1)$$

\begin{df}
For a formal orbifold $(X,P)$ define its genus to be $$g(X,P)=g(X)+\frac{1}{2}\sum_{x\in \supp(P)}\frac{\degram(P(x)/\cK_{X,x})}{[P(x):\cK_{X,x}]}$$
\end{df}

Note that if $O$ is the trivial branch data on $X$ then $g(X,O)=g(X)$.

\begin{df}
Let $f:(Y,Q)\to (X,P)$ be a morphism of formal orbifolds. For $y\in Y$ let $R(y)$ be the field extension $Q(y)/P(f(y))$ (by definition of morphism of formal orbifolds 
$Q(y)\supset P(f(y))$). The function $R$ is called the ramification data of $f$. Note that $f$ is \'etale at $y\in Y$ iff $R(y)$ is the trivial extension. 
The ramification divisor of $f$ is given by $D=\sum_{y\in Y}\frac{\degram(Q(y)/P(f(y)))}{[Q(y):\cK_{Y,y}]}y$. Note that $D$ is a $\QQ$-divisor. 
\end{df} 

\begin{lemma}
 Let $K\subset L \subset M$ be finite extension of local fields and $v_K$, $v_L$ and $v_M$ denote their valuations. Then $\degram(M/K)=\degram(M/L)+\degram(L/K)[M:L]$.
\end{lemma}

\begin{proof}
 Let $\cD(M/K)$ be the different of the local field extension. By \cite[Chapter III, \$4, Proposition 8]{Serre-local.fields} $\cD(M/K)=\cD(M/L)\cD(L/K)$. Now the formula follows from the definition of 
$\degram(M/K)=v_M(\cD(M/K))$ and the fact that $v_M(a)=v_L(a)[M:L]$ for $a\in L$.
\end{proof}

\begin{thm}(Riemann-Hurwitz Formula for formal orbifolds)\label{parabolic-RH}
Let $f:(Y,Q)\to (X,P)$ be a morphism of formal orbifolds of degree $d$. Then $2g(Y,Q)-2=d(2g(X,P)-1)+deg(D)$ where $D$ is the  ramification divisor of $f$. 
\end{thm}

\begin{proof}
This essentially follows from the classical Riemann-Hurwitz formula and the above lemma.
By Riemann-Hurwitz formula, $2g(Y)-2-d(2g(X)-2)=\sum_{y\in Y}\degram(\cK_{Y,y}/\cK_{X,f(y)})$.
So it is enough to show that
$$\sum_{y\in Y}\left(\degram(\cK_{Y,y}/\cK_{X,f(y)})+\frac{\degram(Q(y)/\cK_{Y,y})}{[Q(y):\cK_{Y,y}]}\right)=d\sum_{x\in X}\frac{\degram(P(x)/\cK_{X,x})}{[P(x):\cK_{X,x}]}+\deg(D)$$
But by the above lemma,
\begin{align*}
  LHS =& \sum_{y\in Y}\frac{\degram(\cK_{Y,y}/\cK_{X,f(y)})[Q(y):\cK_{Y,y}]+\degram(Q(y)/\cK_{Y,y})}{[Q(y):\cK_{Y,y}]}\\
      =& \sum_{y\in Y}\frac{\degram(Q(y)/\cK_{X,f(y)})}{[Q(y):\cK_{Y,y}]}
\end{align*}
Also since there is no residue field extension, for any $x\in X$ we have the formula $d=\sum_{y\in Y; f(y)=x}e_y$, where $e_y$ is the ramification index of $f$ at $y$. 
Also note that $e_y=[\cK_{Y,y}:\cK_{X,f(y)}]$.
\begin{align*}
   RHS =& \sum_{x\in X}d\frac{\degram(P(x)/\cK_{X,x})}{[P(x):\cK_{X,x}]}+\sum_{y\in Y}\frac{\degram(Q(y)/P(f(y)))}{[Q(y):\cK_{Y,y}]}\\
      =& \sum_{x\in X}\left(\sum_{y\in Y; f(y)=x} e_y\right)\frac{\degram(P(x)/\cK_{X,x})}{[P(x):\cK_{X,x}]}+\sum_{y\in Y}\frac{\degram(Q(y)/P(f(y)))}{[Q(y):\cK_{Y,y}]}\\
      =& \sum_{y\in Y} [\cK_{Y,y}:\cK_{X,f(y)}]\frac{\degram(P(f(y))/\cK_{X,f(y)})}{[P(f(y)):\cK_{X,f(y)}]}+\sum_{y\in Y}\frac{\degram(Q(y)/P(f(y)))}{[Q(y):\cK_{Y,y}]}\\
      =& \sum_{y\in Y}\frac{[Q(y):P(f(y))]\degram(P(f(y))/\cK_{X,f(y)})}{[Q(y):\cK_{Y,y}]}+\sum_{y\in Y}\frac{\degram(Q(y)/P(f(y)))}{[Q(y):\cK_{Y,y}]}\\
      =& \sum_{y\in Y}\frac{[Q(y):P(f(y))]\degram(P(f(y))/\cK_{X,f(y)})+\degram(Q(y)/P(f(y)))}{[Q(y):\cK_{Y,y}]}\\
      =& \sum_{y\in Y}\frac{\degram(Q(y)/\cK_{X,f(y)})}{[Q(y):\cK_{Y,y}]}]
\end{align*}
\end{proof}

\begin{rmk}
Here we consider morphisms of formal orbifolds which may be ramified (which were not allowed in \cite{parabolic}). Hence the Riemann-Hurwitz formula has the ramification 
term as well. The formula in \cite{parabolic} is a special case of the above theorem. 
\end{rmk}

Recall that if $f:Y\to X$ is $G$-Galois cover of curves then the degree of the ramification divisor has a nice formula in terms of the ramification filtration on the 
Inertia group (Hilbert's different formula). And this leads to a more explicit Riemann-Hurwitz's formula.

\begin{thm}(Riemann-Hurwitz, Hilbert)
 Let $f:Y\to X$ be a $G$-Galois cover branched at $x_1,\ldots,x_n\in X$. For $1\le j\le n$, let $I^j$ be the inertia groups at a point $y_j\in Y$ lying above $x_j$ and let 
$I^j _i$ $i\ge 0$ be the lower ramification filtration on $I^j$. Then
 $$2g(Y)-2=|G|(2g(X)-2)+\sum_{j=1}^n\frac{|G|}{|I^j|}\left(\sum_{i\ge 0} (|I^j_i|-1)\right)$$
\end{thm}

There is a variant of the above statement as well for which we need to define Galois covers of formal orbifolds.

\begin{df}
 Let $f:(Y,Q)\to (X,P)$ be a morphism of formal orbifolds. It is called a $G$-Galois cover for a finite group $G$ if $f:Y\to X$ is a $G$-Galois cover, $Q(y)/P(f(y))$ 
is a Galois extension for all $y\in Y$ and for all $g\in G$, $y\in Y$, the extension $Q(y)/\cK_{X,f(y)}$ is isomorphic to $Q(gy)/\cK_{X,f(y)}$. 
\end{df}

\begin{thm}
 Let $f:(Y,Q)\to (X,P)$ be a $G$-Galois cover branched at $x_1,\ldots,x_n\in X$. For $1\le j\le n$, let $I^j$ be the inertia groups at a point $y_j\in Y$ lying above 
$x_j$ i.e. $I^j=Gal(Q(y_j)/P(x_j))$ and let $I^j _i$ $i\ge 0$ be the lower ramification filtration on $I^j$. Then
 $$2g(Y,Q)-2=|G|(2g(X,P)-2)+\sum_{j=1}^n\frac{|G|}{|I^j|[Q(y_j):\cK_{Y,y_j}]}\left(\sum_{i\ge 0} (|I^j_i|-1)\right)$$
\end{thm}

\begin{proof}
 By Hilbert's different formula we have $\degram(Q(y_j)/P(x_j))=\sum_{i\ge 0} (|I^j_i|-1)$. The result now follows from Theorem \ref{parabolic-RH}.
\end{proof}

\subsection{\'Etale fundamental group}
\begin{df}
Let $\scrX=(X,P)$ be a connected formal orbifold. Let $x$ be a geometric point of $X$ such that its image in $X$ is not in $\supp(P)$. Let $Cov(\scrX)$ be the category of finite \'etale covers of $\scrX$. Let $F_x$ be the functor from $Cov(\scrX)$ to the category of sets given by $F_x(Y,Q)=\Hom_X(x,Y)$ for $[(Y,Q)\to\scrX]\in Cov(\scrX)$. Note that by Corollary \ref{inverse-system} and Lemma \ref{parabolic-etale}, \'etale covers of $\scrX$ form an inverse system. As in the case of \'etale covers of scheme, this functor $F_x$ is also pro-representable by the inverse limit $\tilde \scrX$ of the inverse system $\{(Y_i,Q_i)\in Cov(\scrX)\}_{i\in I}$ of connected \'etale covers of $\scrX$. One fixes a point $\tilde x \in \tilde \scrX$ lying above $x\in \scrX$ (this is equivalent to choosing a point in $Y_i$ over $x$ for all $i\in I$). Then for any $(Y,Q)$ in $cov(\scrX)$, the elements of $F_x(Y,Q)$ is an element of $\Hom_X(\tilde \scrX,Y)$ composed with $\tilde x$. Define the \'etale fundamental group of $\scrX$ to be $Aut(\tilde \scrX/X)$ which is same as $\displaystyle{ \varprojlim_{i\in I} Aut(Y_i/X)}$ over all connected \'etale covers $(Y,Q)\to \scrX$. This group will be denoted by $\pi_1(\scrX)$.
\end{df}

By Lemma \ref{parabolic-etale}, the \'etale fundamental group of $\scrX$ is also $\displaystyle{\varprojlim_{i\in I} Aut(Y_i/\scrX)}$ where $Y_i\to \scrX$ are connected essentially \'etale covers of $\scrX$.
Note that every essentially \'etale cover is dominated by a Galois essentially \'etale cover, hence the inverse limit may be taken over only Galois essentially \'etale covers. 

\begin{pro}\label{pi-is-functorial}
Let $f:\scrY\to \scrX$ be a morphism of formal orbifolds. Then $f$ induces a homomorphism of fundamental group $\pi_1(f):\pi_1(\scrY)\to \pi_1(\scrX)$. Also $\pi_1$ is functor. 
Moreover if $f:\scrY\to\scrX$ is \'etale then $\pi_1(f)$ is injective. 
\end{pro}

\begin{proof}
If $\scrX_1\to \scrX$ is a connected $G$-Galois \'etale cover then its pullback along $\scrY\to \scrX$ is also \'etale by Proposition \ref{etale-pullback}. By Galois theory, the morphism from a connected component $\scrY_1$ of $\scrY\times_{\scrX}\scrX_1$ to $\scrY$ is also Galois \'etale cover with Galois group a subgroup of $G$. Hence there is a natural injective map  
$$i:\Pi=\varprojlim Aut(\scrY_1/\scrY) \to \varprojlim Aut(\scrX_1/ \scrX)$$ 
where the inverse limit is over all \'etale Galois covers $\scrX_1\to \scrX$. As every $\scrY_1\to \scrY$ corresponding to $\scrX_1\to \scrX$ is a Galois \'etale cover of $\scrY$, the  inverse system on the left is a sub-inverse system of all connected Galois \'etale covers of $\scrY$. Hence there is a natural group homomorphism from $q:\pi_1(\scrY)\to \Pi$. Define the map $\pi_1(f)$ to be $i\circ q$. The functoriality of $\pi_1$ follows from the functoriality of pull-backs.

Finally if $f:\scrY\to \scrX$ is \'etale then the two inverse system of \'etale covers of $\scrY$ are the same so $q$ is an isomorphism. Hence $\pi_1(f)$ is injective. 
\end{proof}

\begin{rmk}\label{pi-pi_1}
For a smooth connected projective curve $X$, $\pi_1^{et}(X)=\pi_1(X,O)$ where $O$ is the trivial branch data.
\end{rmk}


\begin{df}
We say that a formal orbifold $(X,P)$ is \emph{geometric} if there exist a connected \'etale cover $(Y,O)\to (X,P)$ where $O$ is the trivial branch data on $Y$. And in this scenario $P$ is called a geometric branch data on $X$. 
\end{df}

Note that if $f:Y\to X$ is a morphism of smooth projective curves then $(X,B_f)$ is a geometric formal orbifold. This is because if $\tilde f:\tilde Y\to X$ is the Galois closure of $f$ then $B_f=B_{\tilde f}$. But Corollary \ref{galois-covers-R_f-parabolic-etale} states that $(X,B_{\tilde f})$ is a geometric formal orbifold.

\begin{pro}\label{product-geometric}
 If $(X,Q_1)$ and $(X,Q_2)$ are geometric formal orbifolds then so is $(X,Q_1Q_2)$.
\end{pro}

\begin{proof}
 It follows from Proposition \ref{fiber-product}, Proposition \ref{etale-pullback} and the observation that composition of \'etale morphisms of formal orbifolds are \'etale.
\end{proof}

\begin{pro}\label{approximate-orbifold}
Let $(X,P)$ be a formal orbifold. There exist $Q \le P$ a (unique) branch data on $X$ such that $(X,Q)$ is a geometric formal orbifold curve and $Q$ is maximal with these properties. The natural homomorphism $\pi_1(X,P)\to\pi_1(X,Q)$ is an isomorphism. 
\end{pro}

\begin{proof}
For a closed point $x\in X$, let $Q(x)$ be the compositum of $B_f(x)$ for all $f:Y\to X$ which are essentially \'etale covers of $(X,P)$. Note that $B_f\le P$ for all such $f$. Note that $\supp(P)$ is finite and $P(x)/\cK_{X,x}$ is a finite extension for all $x\in X$. So there exist finitely many essentially \'etale covers $f_i:Y_i\to X$, say $1\le i\le N$ such that $Q=B_{f_1}B_{f_2}\ldots B_{f_N}$.

Note that by Proposition \ref{product-geometric} if $(X,Q_1)$ and $(X,Q_2)$ are geometric formal orbifolds then so is $(X,Q_1Q_2)$. Also if $Q_1\le P$ and $Q_2\le P$ then so is $Q_1Q_2$. So $Q\le P$ and $(X,Q)$ is a geometric formal orbifold. Also by definition of $Q$, if $(X,Q')$ is a geometric formal orbifold and $Q'\le P$ then $Q'\le Q$. Hence $Q$ is maximal.

Let $f:Y\to (X,P)$ be an essentially \'etale cover. Then $B_f\le P$ and $(X,B_f)$ is a geometric formal orbifold. Hence $B_f\le Q$ which means $f:Y\to (X,Q)$ is also an essentially \'etale cover. Also every essentially \'etale cover of $(X,Q)$ is trivially essentially \'etale cover of $(X,P)$ as $P\ge Q$. Hence $id_X$ induces an isomorphism of $\pi_1(X,P)$ and $\pi_1(X,Q)$.
\end{proof}
 
Given a finite $G$-Galois extension $L/k((t))$ with tame ramification of degree, a Harbater-Katz-Gabber cover (henceforth called HKG-cover) of $\PP^1$ associated to it is a $G$-Galois cover $f:X\to \PP^1$ ramified at only two points $x_0$ and $x_{\infty}$ in $X$ where $f(x_0)=0$ and $f(x_{\infty})=\infty$, $f$ is tamely ramified at 0 and the local extension $\cK_{X,x_{\infty}}/k((t))$ induced by $f$ is same as $L/k((t))$.

The existence of HKG-cover for every finite Galois extension of $k((t))$ ( \cite[Cor 2.4]{Harbater-moduli.of.pcovers}, \cite[Theorem 1.4.1]{Katz-Gabber}) is equivalent to the following statement.
\begin{thm}\label{HKG}
 Let $Q$ be a branch data on $\PP^1$ such that $\supp{Q}=\{0,\infty\}$ and $Q(0)$ is tame and the tame degree of $Q(\infty)/\cK_{\PP^1,\infty}$ is same as 
$[Q(0):\cK_{\PP^1,0}]$. Then $(\PP^1, Q)$ is a geometric formal orbifold.
\end{thm}

\begin{cor}
 Let $Q$ be a branch data on $\PP^1$ such that $\supp(Q)=\{x_0,\ldots,x_m\}$. Let the tame part of $\Gal(Q(x_i)/\cK_{\PP^1,x_i})$ be degree $n_i$. 
Suppose that $n_{i_0}=\max\{n_i: 0\le i\le m\}=\lcm\{n_i: 0\le i\le m, i\ne i_0\}$. Then $(\PP^1,Q)$ is geometric if the $p$-sylow subgroup of $Q(x_{i_0})$ is trivial or there exist $i_1\ne i_0$ 
such that $n_{i_1}=n_{i_0}$.
\end{cor}

\begin{proof}
 We may assume $i_0=0$. For $i\ge 1$ let $P_i$ be the branch data supported at $x_0$ and $x_i$ such that $P_i(x_i)=Q(x_i)$ and $P_i(x_0)$ the tame cyclic extension of degree $n_i$. Note that $n_0=\lcm\{n_i: 1\le i \le m \}$. If $Q(x_0)$ is tame of degree $n_0$ then, it follows that $Q=P_1P_2\ldots P_m$. If there exist $i_1$ such that $n_{i_1}=n_0$ then let $P_0$ be the branch data supported on $x_0$ and $x_{i_0}$ with $P_0(x_0)=Q(x_0)$ and $P_0(x_{i_0})$ be the tame cyclic extension of degree $n_0$. In this case $Q=P_0P_1\ldots P_m$.

Now the result follows from Proposition \ref{product-geometric} and Theorem \ref{HKG}.
\end{proof}

\begin{cor}
 Every purely wild branch data on a smooth projective curve is geometric.
\end{cor}

\begin{proof}

 Let $X$ be a smooth projective curve and $x\in X$. Then there is a constant $d_0$ such that for all $d> d_0$, the linear  system $ {\mathbb P}^N\cong | dx | $ is very ample. Let $H_x$ be the hyperplane defined by $dx=0$ in the linear system. Then $H_x\cap X = \{ x\} $. Let $H\subset H_x$ be a hyperpalne (hence codimension 2 in the linear system) such that $x\notin H$. This can be done
 since the base field is infinite amd most hyperplanes in $H_x$ will not pass through $x$. Let $f$ be the projection ${\mathbb P}^N -H \to {\mathbb P}^1$.
 Then $f$ will restrict to a proper and hence finite morphism from $X\setminus \{x\} \to \Aff^1$ of degree $d$.
 
 Let $P$ be a purely wild branch data on $X$ supported only at $x$. Note that $P(x)/k((t^{-1}))$ is a Galois extension with Galois group $G$ a $p$-group where we identify $\cK_{X,x}$ with $k((t^{-1}))$.
 Also we may choose $d$ to be coprime to $p$ so that the extension $\cK_{X,x}/k((t^{-1}))$ coming from $f$ is linearly disjoint w.r.t. the $p$-group Galois extension $P(x)/k((t^{-1}))$.
 
 Let $Y\to \PP^1$ be the HKG-cover associated to the local extension $P(x)/k((t^{-1}))$. Then $Y\to \PP^1$ is totally ramified at $\infty$ and \'etale elsewhere. The pull-back of $Y\to \PP^1$ along $f$, gives a cover $Y_X\to X$ which is totally ramified at $y$ lying above $x$ and \'etale elsewhere. Moreover the extension $\cK_{Y_X,y}/\cK_{X,x}$ is same as $P(x)/k((t^{-1}))$ because $P(x)/k((t^{-1}))$ and $\cK_{X,x}/k((t^{-1}))$ are linearly disjoint. Hence $P$ is a geometric branch data on $X$. Now the result follows from Proposition \ref{product-geometric}.
\end{proof}

\begin{question}\label{pc-is-geometric}
Which formal orbifolds are geometric formal orbifolds?
\end{question}

\begin{example}
 Let $Q$ be a nontrivial tame branch data on $\PP^1$ with support on a single point. Then $(\PP^1,Q)$ is indeed not geometric since the fundamental group of this formal orbifold is trivial. 
\end{example}

\begin{thm}
  The formal orbifold $(X,P)$ with $P$ tame is geometric iff $X$ is different from $\PP^1$ or $X=\PP^1$ and $|\supp(P)|> 2$ or $X=\PP^1$ and $|\supp(P)|=2$ with same branch data on the two points.  
\end{thm}

\begin{proof}
In characteristic zero the result is a consequence of work of S. Bundagaard, J. Nielsen (\cite{BN}), R.H. Fox (\cite{Fox}) and T.C. Chau (\cite{Chau}). Now use Grothendieck lifting (\cite[Corollary 2.12, XIII, page 392]{SGA1}) to obtain the result in characteristic $p$.
\end{proof}

\begin{pro}\label{pc-is-almost-geometric}
Every formal orbifold is dominated by a geometric formal orbifold, i.e. given a formal orbifold $\scrX$ there exist a morphism of formal orbifolds $\scrY\to \scrX$ such that $\scrY$ is a geometric formal orbifold. 
\end{pro}

\begin{proof}
Let $(X,P)$ be a formal orbifold. Since $P$ has finite support, in view of Proposition \ref{product-geometric}, we may assume $P$ is supported at only one point $x \in X$. Then $k(X)\subset \cK_{X,x}$ and $P(x)/\cK_{X,x}$ is a finite extension. There exist a finite extension $L/k(X)$ with $L\subset P(x)$ such that $L\cK_{X,x}=P(x)$. Let $Y$ be the normalization of $X$ in $L$. Then $f:Y\to X$ is finite cover of projective curves and $B_f(x)$ clearly contains $P(x)$. Hence the orbifold $(X,B_f)$ dominates $(X,P)$. 
\end{proof}

\begin{thm}\label{domination-pi-surjective}
Let $X$ be a smooth projective connected curve. Let $P\ge P'$ be branch data on $X$. Then there is a natural epimorphism $\pi_1(X,P)\to \pi_1(X,P')$ induced from the 
$id_X$ (which by Remark \ref{id-morphism} is a morphism of formal orbifolds).
\end{thm}

\begin{proof}
Every essentially \'etale cover of $(X,P')$ is also an essentially \'etale cover of $(X,P)$. Hence the inverse system of connected essentially \'etale covers of $(X,P')$ is a sub-inverse system of connected essentially \'etale covers of $(X,P)$. Hence the natural map $\pi_1(X,P)\to \pi_1(X,P')$ is an epimorphism.
\end{proof}

\begin{cor}\label{pi_1-finite-rank}
 The \'etale fundamental group of a connected formal orbifold is of finite rank.
\end{cor}

\begin{proof}
Note that $\pi_1^{et}(Y)$ is of finite rank for a connected smooth projective curve $Y$ (\cite[Theorem 3.8, X, page 283]{SGA1}). By Proposition \ref{approximate-orbifold}, we may assume that the formal orbifold $(X,P)$ is geometric. Hence there exist a finite \'etale Galois cover $(Y,O)\to (X,P)$. Now the result follows from Proposition \ref{pi-is-functorial} and Remark \ref{pi-pi_1}.
\end{proof} 

\begin{thm}\label{pi_1-approx}
Let $C$ be a smooth affine connected curve and $X$ its smooth completion and $S=X\setminus C$. Then $\pi_1^{et}(C)=\varprojlim \pi_1(X,P)$ where $P$ runs over all branch data on $X$ with $\supp(P)\subset S$.
\end{thm}

\begin{proof}
Every \'etale cover of $D\to C$ extends to a essentially \'etale cover $Y\to (X,P)$ for some branch data $P$ on $X$ with $\supp(P)\subset S$ where $Y$ is the smooth 
completion of $D$.
\end{proof}

\begin{rmk}
 The inverse limit in the above may be taken over $P$ such that $\supp(P)=S$, since this will be cofinal to all $P$ with $\supp(P)\subset S$.
\end{rmk}

Let $f:(Y,Q)\to (X,P)$ be a morphism of formal orbifolds. An intermediate cover $(Z,R)\to (X,P)$ is said to be \emph{maximal \'etale subcover} of $(Y,Q)\to (X,P)$ if $(Z,R)\to (X,P)$ is \'etale and for any other $(Z',R')\to (X,P)$ intermediate \'etale subcover of $(Y,O)\to (X,P)$, the cover $(Z',R')\to (X,P)$ is an intermediate subcover of $(Z,R)\to (X,P)$. In other words the morphism $Z\to X$ factors through $Z'\to X$.

\begin{pro}
Let $f:(Y,Q)\to (X,P)$ be a morphism of formal orbifolds. Let $H=\im[\pi_1(f):\pi_1(Y,Q)\to \pi_1(X,P)]$. Then $H$ is a finite index open subgroup of $\pi_1(X,P)$. Let $L$ be the extension of $k(X)$ corresponding to $H$ (i.e. $L=[k(X)^{P,un}]^H$ where $k(X)^{P,un}$ is the compositum of $k(W)$'s where $(W,B)$ vary over all finite \'etale covers of $(X,P)$). Let $g:Z\to X$ be the normalization of $X$ in $L$ and $R=g^*P$ then $(Z,R)$ is the maximal \'etale cover of $(X,P).$
\end{pro}

\begin{proof}
By construction note that $g:Z\to X$ is an essentially \'etale cover of $(X,P)$. Hence by Lemma \ref{parabolic-etale}, $(Z,R)\to (X,P)$ is \'etale. Moreover it is an intermediate cover of $(Y,Q)\to (X,P)$. So enough to show that it is maximal and again by Lemma \ref{parabolic-etale} it is enough to show that if $h:Z'\to X$ is an essentially \'etale cover of $(X,P)$ dominated by $Y$ then $Z$ dominates $Z'$.

Let $R'=h^*P$ be the branch data on $Z'$. Then we have morphisms $(Y,Q)\to^p (Z',R')\to^h (X,P)$ with $h\circ p=f$. By functoriality of $\pi_1$, $\pi_1(f)=\pi_1(h)\circ\pi_1(p)$. Hence $\image(\pi_1(h))\supset \image(\pi_1(f))=H$. Hence by Galois theory $Z$ dominates $Z'$.
\end{proof}

\subsection{Anabelian type conjecture}

In \cite[Theorem 8.6]{Tamagawa} it was shown that if $C$ is a smooth curve different from an ordinary elliptic curve over $k=\bar\FF_p$. Then there are only finitely many smooth curves $X$ such that $\pi_1(X)\cong\pi_1(C)$. This result generalizes to geometric formal orbifold setup as well.

\begin{thm}
Let $\scrC$ be a geometric formal orbifold over $k=\bar\FF_p$ such that the underlying curve has genus at least 2. Then there are only finitely many smooth projective curve $X$ such that $\pi_1(X)\cong\pi_1(\scrC)$. 
\end{thm}

\begin{proof}
 Let $X$ be a smooth projective curve such that $\phi:\pi_1(\scrC)\to \pi_1(X)$ is an isomorphism. There exist Galois \'etale morphism  $(D,O)\to \scrC$. Then $\pi_1(D)$ is a finite index normal subgroups of $\pi_1(\scrC)$. Let $N=\phi(\pi_1(D))$ and $Y$ be the \'etale cover of $X$ such that $\pi_1(Y)=N$. Note that $N$ and $\pi_1(D)$ are isomorphic and by Riemann-Hurwitz formula genus of $D$ is at least 2. So there are only finitely many possibilities for $Y$ by Tamagawa's result (\cite[Theorem 8.6]{Tamagawa}). Let $G=\pi_1(X)/N$ then $Y$ is a $G$-cover of $X$. Now using Theorem of de Franchis (\cite{Kani}), for each $Y$ there are only finitely many smooth curves $Z$ and coverings $Y\to Z$ of degree $|G|$. Hence there are only finitely many possibilities for $X$. 
\end{proof}

\section{Orbifold bundles} \label{sec:orbifolds}

Given a curve $X$ over a characteristic zero field and a vector bundle $V$ on $X$ one associates 
a flag structure on $V_x$ for finitely many points $x\in X$. Further one also assigns real 
weights $0< \alpha_1< \ldots < \alpha_r\le 1$ at this flag of length $r$ on $V_x$. The vector 
bundle $V$ together this data is called a 
parabolic bundle on $X$. In the terminology of \cite{parabolic}, when the weights $\alpha_i=p_i/q_i$ are rational numbers, then 
$V$ can be thought of as a \emph{vector bundle on the parabolic curve} $(X,P)$ where $P(x)$ 
is a cyclic extension of degree $\lcm\{q_i\}$. Note that in characteristic zero the field 
extension is purely determined by the extension degree. Hence $P(x)$
may be thought of as a number  rather than a field extension.

Given a geometric parabolic curve $(X,P)$ there exist a Galois etale cover of parabolic curves 
$f: (Y,O)\to(X,P)$ with Galois group $\Gamma$. Then it is known that the category of parabolic
vector bundles on $X$ with rational weights  is equivalent to the category 
of orbifold bundles on $X$, i.e. vector bundles on $Y$ together with a lift of $\Gamma$-action 
\cite{Biswas}, \cite[Lemma 4.4]{parabolic} for any fixed $\Gamma$ and $Y$. Such bundles are called 
\emph{$\Gamma$-bundles} on $Y$.  Morphisms between $\Gamma$-bundles is a $\Gamma$-equivariant morphism of vector bundles.

When the base field has characteristic $p>0$, the analogous category of orbifold bundle on a smooth projective curve $X$ can be defined. But it not yet clear that if $V$ is vector bundle on $X$ what extra data is required to make $V$ into an analogue of a parabolic vector 
bundle with rational weights on the curve $X$. 

We will define vector bundles on the geometric formal orbifold curve $(X,P)$ to be a vector bundle  on $Y$ where  $(Y,O)\to (X,P)$ 
is a Galois etale cover of the formal orbifold $(X,P)$ with Galois group $\Gamma$. We will show that 
this category is independent of the chosen Galois cover $(Y,O)$. These will be the building blocks for the category of 
orbifold bundles on $X$.

We begin by studying $\Gamma$-bundles on $Y$ where $f:Y\to X$ is a $\Gamma$ cover of smooth projective curves.

\begin{df}\label{G-bundle-defn}
 Let $V$ be a vector bundle on $Y$. Let $\mu:\Gamma\times Y \to Y$ be the morphism induced from the action of $\Gamma$ on $Y$ and $p:\Gamma\times Y \to Y$ be the second projection. Giving the structure of a $\Gamma$-bundle on $V$ is equivalent to giving an isomorphism  $\lambda:\mu^*V\to p^*V$ ``compatible'' with the $\Gamma$-action, i.e. for $g\in \Gamma$, if $\lambda_g$ denote $\lambda|_{g\times Y}:g^*V\to V$ then for all $g,h\in \Gamma$ 
 \begin{equation}\label{eq:co-cycle}
  \lambda_{gh}=\lambda_h\circ h^*\lambda_{g}. 
 \end{equation}

 Moreover let $W$ be another $\Gamma$-bundle and $\phi:V\to W$ be a morphism of of vector bundles. Then $\phi$ is a morphism of $\Gamma$ bundles if the following diagram commute:
 \[
 \xymatrix{
    \mu^*V\ar[d]^{\mu^*\phi}\ar[r]  & p^*V\ar[d]^{p^*\phi}\\
    \mu^*W\ar[r]                             & p^*W
    }
  \]
\end{df}

\begin{lemma}\label{action-on-pull-back}
 Let $V$ be an vector bundle on $X$ and $f:Y\to X$ be a Galois cover of smooth projective curves 
with Galois group $\Gamma$. Then $f^*V$ is naturally a $\Gamma$-bundle on $Y$. In fact the action 
on $f^*V$ can be identified as the natural action of $\Gamma$ on ${\mathcal O}_Y$  together with 
the isomorphism $f^*V\cong f^{-1}(V)\otimes {\mathcal O}_Y$. i.e., $\gamma\cdot(v,f)
= (v,\gamma\cdot f)$ for $v\in f^{-1}(V)$ and $f\in {\mathcal O}_Y$. 
\end{lemma}

\begin{proof}
 Let $\mu$ and $p$ be the action and projection map from $\Gamma\times Y\to Y$. Then $f\circ \mu=f\circ p$. Hence by functoriality of pull-backs, we have $\mu^*f^*V\cong p^*f^*V$ and the isomorphism is compatible with the $\Gamma$-action. 
\end{proof}

In fact this is true even if we have towers of Galois covers.

\begin{lemma}\label{action-on-pull-back-G-bundle}
 Let $f:Z\to Y$ and $b:Y\to X$ be $H$-Galois and $G$-Galois covers respectively such that $b\circ f$ is a $\Gamma$-cover. Then for any $G$-bundle $V$ on $Y$, $f^*V$ has a natural $\Gamma$-bundle structure. Moreover, if $\phi:V\to W$ is a morphism of $G$-bundles then $f^*\phi$ is a morphism of $\Gamma$-bundles. 
\end{lemma}

\begin{proof}
 First we note that $H$ is a normal subgroup of $\Gamma$, $G=\Gamma/H$ and let $q:\Gamma\to G$ be the group epimorphism. Let $\mu$ and $p$ (with appropriate subscript) denote the action and projection maps respectively. Since $V$ is a $G$-bundle, there is an isomorphism $\lambda:\mu_G^*V\cong p_G^*V$ compatible with the $G$-action.
 
 We have the following two cartesian diagrams:

 \begin{equation}\label{eq:projection}
  \xymatrix{
    \Gamma\times Z \ar[r]^{\tilde q}\ar[d]^{\beta}  &  G\times Z\ar[r]^{\tilde p_G}\ar[d]   &  Z\ar[d]^f\\
    \Gamma\times Y \ar[r]_q                         & G\times Y \ar[r]_{p_G}              & Y}
 \end{equation}
 
 \begin{equation}\label{eq:action}
  \xymatrix{
    \Gamma\times Z \ar[r]^{\tilde q}\ar[d]^{\beta}  &  G\times Z\ar[r]^{\tilde \mu_G}\ar[d]  &  Z\ar[d]^f\\
    \Gamma\times Y \ar[r]_q                         & G\times Y \ar[r]_{\mu_G}             & Y}
 \end{equation}

 From \eqref{eq:projection} we obtain that $f\circ\tilde p_G\circ\tilde q=p_G\circ q \circ \beta$. Also note that $\tilde p_G\circ\tilde q=p_{\Gamma}$. Hence $f\circ p_{\Gamma}=p_G\circ q \circ \beta$.
 
 Similarly from \eqref{eq:action}, we get $f\circ\tilde \mu_G\circ\tilde q=\mu_G\circ q \circ \beta$. Also since the action of $\Gamma$ on $Y$ factors through $G$, we have the following commutative diagram:
 
 \begin{equation}
  \xymatrix{
    \Gamma\times Z \ar[rr]^{\mu_{\Gamma}}\ar[d]^{\beta}  &                           &  Z\ar[d]^f\\
    \Gamma\times Y \ar[r]_q                              & G\times Y \ar[r]_{\mu_G}  & Y}
 \end{equation}
 
 And this gives us $f\circ \mu_{\Gamma}=\mu_G\circ q\circ\beta$. The pull-back of $\lambda$ along $q\circ \beta$ induces the required isomorphism $$\mu_{\Gamma}^*f^*V\cong\beta^*q^*\mu_G^*V\xrightarrow{\sim}\beta^*q^*p_G^*V\cong p_{\Gamma}^*f^*V$$.
 
 The statement about pullback of morphisms also follow from the identities $f\circ p_{\Gamma}=p_G\circ q \circ \beta$ and $f\circ \mu_{\Gamma}=\mu_G\circ q\circ\beta$ and functoriality of pull-backs.
\end{proof}

\begin{lemma}\label{invariant-pushfoward-etale}
 Let $f:Y\to Y'$ be a $G$-Galois cover of smooth complete curves. Let $V$ be a vector bundle on $Y'$ and $W$ be a $G$-bundle on $Y$. Then $(f_*f^*V)^G=V$. Further if $f$ is \'etale,  then $f^*[(f_*W)^G]=W$.
\end{lemma}

\begin{proof}
First one notices that $\cO_Y$ is a $G$-vector bundle on $Y$. This implies that $G$ acts on the vector bundle $f_*\cO_Y$ as automorphisms of vector bundles on $Y'$. Moreover the completion stalks of 
$f_*\cO_Y$ at $x\in Y'$ can be identified as $\hat \cO_{Y',x}\otimes k[G]$ where the $G$ action on this
tensor product is given by the action on the second factor since $G$ acts trivially on $\cO_{Y'}$.
Hence $\cO_{Y'}$ is a $G$-invariant subspace of $f_*\cO_Y$ and $G$ acts faith fully on $Y$, it follows
that $(f_*\cO_Y)^G=\cO_{Y'}$. Now if $V$ is any vector bundle on  $Y'$, then by projection formula, 
we have $(f_*f^*V)=V\otimes f_*f^*\cO_{Y'}$. Moreover the action of $G$ on 
$V\otimes f_*f^*\cO_{Y'}$ is by the action on the second factor and acting trivially on $V$ as 
action on $f^{-1}(V)$ is trivial by Lemma~\ref{action-on-pull-back}. Hence the invariance is obtained by  
$$(f_*f^*V)^G=V\otimes (f_*f^*\cO_{Y'})^G \cong V\otimes \cO_{Y'}\cong V$$

For the second part, 
note that $f:Y\to Y'$ is a $\Gamma$-Galois \'etale morphism. The descent datum for locally free sheaves on $Y$ is exactly the co-cycle condition \eqref{eq:co-cycle} which makes a locally free sheaf into a $\Gamma$-bundle. Hence by \'etale descent any $\Gamma$-bundle $W$ on $Y$ descends to a vector bundle $V$ on $Y'$. 
Now one can use the first part to see that this $V$ is isomorphic to the invariant 
direct image $(f_*W)^G$. Hence it follows that $f^*[(f_*W)^G]\cong W$.
\end{proof}

\begin{df}
 Let $(X,P)$ be a geometric formal orbifold curve. Fix an \'etale Galois cover $(Y,O)\to (X,P)$ with Galois group $\Gamma$. A \emph{vector bundle on the geometric formal orbifold} $(X,P)$ is a $\Gamma$-bundle 
on $Y$. Morphisms of vector bundles on $(X,P)$ is a $\Gamma$-equivariant sheaf homomorphism between the corresponding $\Gamma$-bundles on $Y$. This category will be denoted by $Vect(X,P)$.
\end{df}

\begin{pro}\label{parabolic-bundles-well-defined}
 The category of vector bundles on $(X,P)$ does not depend on the choice of the \'etale Galois cover.
\end{pro}

\begin{proof}
Let $(Y_i,O)\to (X,P)$ for $i=1,2$ be \'etale Galois covers with Galois group $\Gamma_i$. Let $(Y,O)\to (X,P)$ be an \'etale Galois cover with Galois group $\Gamma$ which dominates $Y_i$'s (for instance $Y$ could be a connected component of the normalized fiber product). Hence the cover $f_i:Y\to Y_i$ is \'etale Galois with Galois group $G_i$ where $G_i$ is the kernel of the natural surjection $\Gamma\to \Gamma_i$. 

We shall show that the category of $\Gamma$-bundles on $Y$ is equivalent to the category of $\Gamma_1$-bundles on $Y_1$ (similarly it follows that it is equivalent to category of $\Gamma_2$-bundles on $Y_2$).

If $V$ is a $\Gamma_1$-bundle on $Y_1$, then $f^*V$ is a $\Gamma$-bundle on $Y$ by Lemma \ref{action-on-pull-back-G-bundle}. Since $G_1$ is a subgroup of $\Gamma$, given a $\Gamma$-bundle $W$ on $Y$, it is also a $G_1$-bundle on $Y$ and $V:=[f_*W]^{G_1}$ is a vector bundle on $Y_1$ by \'etale Galois descent. The $\Gamma$-action on $W$ induces a $\Gamma_1=\Gamma/G_1$-action on $V$ which is compatible with $\Gamma_1$-action on $Y_1$. Hence $V$ is $\Gamma_1$-bundle on $Y_1$. Finally by Lemma \ref{invariant-pushfoward-etale} the two categories are equivalent.
\end{proof}

\begin{thm}
Let $P_1\ge P_2$ be geometric branch data on $X$. There exists a functor $Vect(X,P_2)\to Vect(X,P_1)$ which is an embedding.
\end{thm}

\begin{proof}
Let $(Y_i,O)\to (X,P_i)$ be an \'etale $\Gamma_i$-cover. Let $Y$ be a connected component of the normalization of $Y_1\times_X Y_2$. As $P_1\ge P_2$, $Y_2\to X$ is an essentially \'etale cover of $(X,P_1)$. Hence $(Y,O)\to (X,P_1)$ is an \'etale Galois cover with Galois group 
$\Gamma$  which dominates $Y_i$. Let $G_i=ker(\Gamma\onto \Gamma_i)$. Let $f_i:Y\to Y_i$ be the cover obtained from the projection maps. Then $f_i$  are $G_i$-Galois cover and $f_1$ is \'etale. Hence the category of $\Gamma$-bundles on $Y$ is equivalent to the category of $G_1$-bundles on $Y_1$ (by  Proposition \ref{parabolic-bundles-well-defined} and the two categories are equivalent to the category of vector bundles on the geometric formal orbifold $(X,P_1)$. Let $V$ be a $\Gamma_2$-bundle on $Y_2$ then $f_2^*V$ is naturally a $\Gamma$-bundle on $Y$ (Lemma \ref{action-on-pull-back-G-bundle}) and the morphisms of $\Gamma_2$-bundles also pullback to morphisms of $\Gamma$-bundles on $Y$. Hence $f_2^*$ is a functor from $Vect(X,P_2)\to Vect(X,P_1)$. This functor is an embedding follows from Lemma \ref{invariant-pushfoward-etale}.      
\end{proof}

Hence we may think of $Vect(X,P_2)$ as a (fully faithful) subcategory of $Vect(X,P_1)$.

\begin{df}
 An {\em orbifold bundle} on a smooth projective curve $X$ is a vector bundle on $(X,P)$ for some geometric branch data $P$ on $X$. By the above theorem if $V\in Vect(X,P)$ then $V\in Vect(X,P')$ for all geometric $P'\ge P$. Let $V$ and $W$ be two orbifold bundles on $X$ then $V\in Vect(X,P)$ and $W\in Vect(X,Q)$ for some geometric branch data $P$ and $Q$. Then $PQ$ is a geometric branch data on $X$ and $V,W\in Vect(X,PQ)$. Morphisms from an orbifold bundle $V$ to $W$ are morphisms in the category $Vect(X,PQ)$. So we have a category of orbifold bundles on $X$ which we will denote by $PVect(X)$.
\end{df}

Note that this category has direct sums, tensor products and duals. 

\begin{rmk}
 Let $(Y,O)\to (X,P)$ and $(Y',O)\to (X,P)$ be \'etale covers with Galois group $\Gamma$ and $\Gamma'$ such that $Y\to (X,P)$ factors through $Y'$ and $f:Y\to Y'$ be the factor morphism. Let $V'$ be a $\Gamma'$-bundle on $Y'$ and $V=f^*V'$ then $\deg(V)=|\Gamma/\Gamma'|\deg(V')$. 
\end{rmk}

We define the \emph{degree} of a vector bundle $(Y,\Gamma,V)$ on a geometric formal orbifold $(X,P)$ to be $\deg(V)/|\Gamma|$ and the orbifold slope $\parslope(Y,\Gamma,V)$ as the degree divided by the rank of the bundle. Note that $(Y,\Gamma,W)$ is a subbundle of $(Y,\Gamma,V)$ 
iff $W$ is a $\Gamma$-invariant subbundle of $V$. We say that a vector bundle on a geometric formal orbifold $(X,P)$  is {\em semistable} bundle if  it is represented by $(Y,\Gamma,V)$ such that for any subbundle $(Y,\Gamma,W)$, $\parslope(Y,\Gamma,W)\le\parslope(Y,\Gamma,V)$. If a bundle on $(X,P)$ is represented by $(Y,\Gamma,V)$, then
we can define the Frobenius pull back of $(Y,\Gamma,V)$ to be the bundle $(Y,\Gamma,F_Y^*V)$ where $F_Y$ denotes the Frobenius morphism of $Y$. Here we note that if $V$ is a $\Gamma$ bundle on $Y$, then so is 
${F_Y^n}^*V$ for all $n\geq 1$. We call $(Y,\Gamma,V)$ {\em strongly semistable}  bundle  on $(X,P)$ if
all its Frobenius pull backs are semistable.  

\begin{lemma}
A vector bundle $(Y,\Gamma,V)$ on $(X,P)$ is semistable (resp. strongly semistable) iff $V$ is semistable (resp. strongly semistable) vector bundle on $Y$.
\end{lemma}

\begin{proof}
 Harder-Narshimshan filtration is unique, hence subbundles in the filtration of a $\Gamma$-bundle is $\Gamma$-invariant.
\end{proof}

\subsection{Tannakian category and fundamental groups}

\begin{thm}\label{tannakian}
Let $X$ be a smooth projective curve and $x$ be any geometric point on $X$. For a geometric formal orbifold 
$(X,P)$ the category of strongly semistable bundles of degree $0$ on $(X,P)$, denoted
by $Vect^0(X,P)^{ss}$, together with the fiber functor associated to $x$ is a neutral Tannakian category. The same is true for the subcategory $Nori(X,P)$ of essentially finite bundles and the subcategory $Etale(X,P)$ of bundles associated to finite \'etale covers. 
\end{thm}

\begin{proof}
It is well known that strongly semistable vector bundles of degree $0$ on a smooth projective connected
curve form a neutral Tannakian category (cf. [BPS]). Now the theorem will follow by noting that if you 
have a  $\Gamma$-morphism of $\Gamma$-bundles, then both kernel and cokernel are again $\Gamma$-bundle. 
The last staement also follows similarly since the kernel and cokernels exist in these subcategories. 
\end{proof}

\begin{rmk}\label{Nori-pi_1}
We denote the affine group scheme duals to the Tannakian categories $Vect^0(X,P)^{ss}$, $Nori(X,P)$ and $Etale(X,P)$ by $\mathcal{M}(X,P)$, $\mathcal{N}(X,P)$ and $\pi_1^{alg}(X,P)$ respectively. Like in the projective curve case $\pi_1^{alg}(X,P)$ is isomorphic to the \'etale fundamental group  $\pi_1(X,P)$. Note that $\mathcal{N}(X,O)$ is the Nori fundamental group of $X$. Also note that $Vect^0(X,P)^{ss}\supset Nori(X,P)\supset Etale(X,P)$, so $\pi_1(X,P)$ is a homomorphic image of $\mathcal{N}(X,P)$ and $\mathcal{N}(X,P)$ is a quotient of $\mathcal{M}(X,P)$.
\end{rmk}

Let us denote by $PVect^0(X)^{ss}$, the category of strongly semistable orbifold bundles of degree $0$
on $X$. Then we have:

\begin{thm} $PVect^0(X)^{ss}$, together with the fiber functor associated to $x$ is a neutral Tannakian category. If we denote the affine group scheme dual to this Tannakian category by $\mathcal{M}(X)$, then 
$\mathcal{M}(X,P)$ is a quotient of $\mathcal{M}(X)$ for all geometric branch data $P$ on $X$.
\end{thm}

\subsection{Projection formula}

Let $f:(X_1,P_1)\to (X_2,P_2)$ be an \'etale cover. 
Let $(Y_2,O)\to (X_2,P_2)$ be a Galois \'etale cover, 
with Galois group $\Gamma$. Then define  the direct image of vector bundles on the formal orbifold 
$(X_1,P_1)$ as follows. Consider the normalization of the fiber product 
$Y_1:= \widetilde{X_1\times _{X_2} Y_2}$. Then $Y_1\to Y_2$ is an \'etale cover, though need not be connected.
Moreover $(Y_1,O) \to (X_1,P_1)$ is an \'etale Galois morphism of formal orbifolds with Galois group 
$\Gamma$. Hence every vector bundle on $(X_1,P_1)$ is represented by a $\Gamma$-bundle on $Y_1$. Then the 
usual direct image on $Y_2$ is clearly a $\Gamma$-bundle, hence defines a vector bundle on the geometric formal orbifold
$(X_2,P_2)$. We define this as the \emph{direct image} of vector bundles in $Vect(X_1,P_1)$ and denote it by $\hat{f}_*$. 

Consider  the Galois \'etale cover $(Y_2,O)\to \scrX_2$ with Galois group  $\Gamma$  and the 
normalized fiber product $Y_1$ as above. Then $(Y_1,O)\to (X_1,P_1)$ is a Galois cover with Galois group 
$\Gamma$. The pull back under $Y_1\to Y_2$ of a $\Gamma$ 
bundle on $Y_2$  defines a $\Gamma$ bundle on $Y_1$, hence it defines a vector bundle on the geometric formal orbifold 
$(X_1,P_1)$. We define this to be the pull back of vector bundles on geometric formal orbifold and denote it by 
$\hat{f}^*$.

Recall that the tensor product of two vector bundles on a geometric formal orbifold $(X,P)$ is defined to be the usual
tensor product of vector bundles on an \'etale Galois cover with trivial branch data where both 
bundles are represented. We denote it by $\hat\otimes$.  

Now we can state our next result which is an analogue of the projection formula.

\begin{thm} \label{projection-formula}
Let $f:(X_1,P_1)\to (X_2,P_2)$ be an \'etale cover. Assume $\mathcal{V}$ is a vector bundle on the geometric formal orbifold
$(X_2,P_2)$ and $\mathcal F$ is a vector bundle  on the geometric formal orbifold $(X_1,P_1)$. Then we have 
$$\hat{f}_*(\hat{f}^*V\hat\otimes \mathcal{F}) \cong V\hat\otimes \hat{f}_*(\mathcal{F})$$
\end{thm}

\begin{proof} Choose an \'etale Galois cover $(Y_2,O)$ of $(X_2,P_2)$ with a $\Gamma$-bundle $V$ representing the given vector bundle on $(X_2,P_2)$. Also choose a $\Gamma$-bundle $F$ on the
normalization of the fiber product $Y_1$ representing the given bundle on $(X_1,P_1)$. Then the 
above theorem is equivalent to the statement that we have isomorphism on $Y_2$ of $\Gamma$-bundles 
$$\tilde{f}_*(\tilde{f}^*(V) \otimes F) \cong  V\otimes \tilde{f}_*(F)$$
where $\tilde{f}$ is the map from the fiber product, $Y_1\to Y_2$.
\end{proof}

\end{document}